\documentclass[12pt]{article}

\usepackage[left=1in,right=1in,top=1in,bottom=1in]{geometry}

\usepackage{subfigure}
\usepackage{graphicx}
\usepackage{amsfonts}
\usepackage{url}
\usepackage{amsmath,amssymb, amsthm}
\usepackage{enumerate}
\usepackage{float}
\usepackage{mhchem, cite}
\usepackage{bbm}
\usepackage[normalem]{ulem}

\usepackage{mathtools}

\usepackage[colorlinks]{hyperref}

\newtheorem{definition}{Definition}

\newtheorem{remark}{Remark}

\newtheorem{theorem}{Theorem}
\newtheorem{corollary}{Corollary}
\newtheorem{condition}{Condition}

\theoremstyle{definition}
\newtheorem{example}{Example}

\newcommand{\RR}{{\mathbb R}}
\newcommand{\ZZ}{{\mathbb Z}}

\newcommand{\Sp}{\ensuremath{\mathcal{S}}}
\newcommand{\R}{\ensuremath{\mathcal{R}}}
\newcommand{\C}{\ensuremath{\mathcal{C}}}

\title{Non-explosivity of stochastically modeled reaction networks that are complex balanced}

\begin{document}

\author{
David F.\ Anderson\footnotemark[1] \footnotemark[3]
\and
Daniele Cappelletti\footnotemark[1] \footnotemark[4]
\and
Masanori Koyama\footnotemark[2]
\and
Thomas G.\ Kurtz\footnotemark[1]
}

\footnotetext[1]{Department of Mathematics, University of Wisconsin-Madison, USA.}
\footnotetext[2]{Department of Mathematical Science, Ritsumeikan University, Japan.}
\footnotetext[3]{Supported by NSF-DMS-1318832 and Army Research Office grant W911NF-14-1-0401.}
\footnotetext[4]{\textit{Corresponding author}. Email address \texttt{cappelletti@math.wisc.edu}.}

\maketitle

\begin{abstract}
	We consider stochastically modeled reaction networks and prove that if a constant solution to the Kolmogorov forward equation decays fast enough relatively to the transition rates, then the model is non-explosive.  In particular, complex balanced reaction networks are non-explosive.
\end{abstract}

\section{Introduction}

Stochastic reaction networks are mathematical models that describe the time evolution of the counts of interacting objects. Despite their general formulation, these models are primarily associated with biochemistry and cell biology, where they are utilized to describe the dynamics of systems of interacting proteins, metabolites, etc. \cite{cornish:enz-kinetics, ingram:nonidentifiability, karlebach:modelling,Paulsson2004}. Other applications include ecology \cite{may:stability}, epidemiology \cite{anderson:infectious}, and sociology \cite{weidlich:concepts, peschel:socio, peschel:predator}. A more recent development concerns the design of stochastic reaction networks that are able to perform computations  \cite{winfree:program, winfree:dna, doty:tile, doty:prob}.

Stochastically modeled reaction networks are typically utilized in biochemistry when at least some of the species in the model have low abundances, perhaps of order $10^2$ or fewer. In this situation, the randomness that arises from the interactions of the molecules cannot be ignored, and the counts of the molecules of the different chemical species are  modeled according to a discrete space,  continuous-time Markov chain. If   more molecules are present, perhaps between order $10^2$ and  $10^3$, a diffusion model is often utilized, which describes the system in terms of a stochastic differential equation \cite{AK2015}.  When abundances are larger,  a deterministic model is often chosen, which describes the dynamics   by means of a system of ordinary differential equations \cite{AK2011,AK2015}.  

When modeling biochemical systems, it may happen that the designed model exits every compact set in a finite time. If the model utilized is a deterministic system of differential equations, such behavior corresponds to the presence of a vertical asymptote. When this happens in a stochastic model, either in a continuous-time Markov chain or in a system of stochastic differential equations, the event takes the name of ``explosion''. Since in nature no quantity can reach infinity in a finite time, the presence of a vertical asymptote or of an explosion implies that the model fails to accurately describe the system of interest, or at least that an uncontrolled growth in the system causes it to exit the region where the model is valid. A famous example of this situation is given by the chemotaxis models discussed in \cite{herrero, chertok, childress}.

In \cite{ACK2010}, constant solutions to the forward Kolmogorov equation were constructed for an important class of continuous-time Markov chain models of chemical reaction networks, specifically those that are complex balanced. The results presented in this paper are a response to a question raised by Eduardo Sontag regarding whether or not the processes  in this class could have explosions. The answer, given in this paper, is no. Specifically, we give sufficient conditions on the constant solutions of the forward Kolmogorov equation of a stochastic reaction network that imply non-explosiveness. We then show that these conditions are satisfied by the complex balanced systems studied in \cite{ACK2010}.

Continuous time Markov chain models, say with countable state space $E\subset \ZZ^d$, are usually specified by giving transition intensities $0\leq q(x,y)<\infty$.  The intuitive interpretation of these intensities is that the corresponding process $(X_t)_{t\geq0}$ should satisfy
\begin{equation}
P( X_{t+h}=y|{\cal F}_t^X)=q(X_t,y)h+o(h),\label{tranint}
\end{equation}
where ${\cal F}_s^X$ represents all information about the process $(X_t)_{t\geq0}$ up to time $s$ (that is, $\{{\cal F}_t^X\}$ is the filtration generated by $(X_t)_{t\geq0}$).  One way of making this intuition precise is to derive the corresponding  forward/master equation for the one-dimensional distributions $P_t(x)=P(X_t=x)$, $t\geq 0$, $x\in E$, that is,
\begin{equation}\label{feq0}
\frac d{dt}P_t(x)=\sum_{y\in E\setminus \{x\}}P_t(y)q(y,x)-\sum_{y\in E\setminus \{x\}}P_t(x)q(x,y),
\end{equation}
for a specified initial distribution $P_0$, where the derivative at $t=0$ is interpreted as a right derivative.

We can give conditions for explosions in terms of what is called the minimal solution of the system \eqref{feq0}: the model is non-explosive if and only if the minimal solution of \eqref{feq0} satisfies $\sum_{x\in E} P_t(x)=1$ for all $t\geq0$, provided that $P_0(x)>0$ for all $x\in E$ \cite{AK2015}. However, those conditions are not helpful if what we are given is a constant solution of \eqref{feq0}, that is, a probability distribution $\pi$ satisfying 
\begin{equation}
0=\sum_{y\in E\setminus \{x\}}\pi (y)q(y,x)-\sum_{y\in E\setminus \{x\}}\pi (x)q(x,y).\label{cnstsol}
\end{equation}

It is worth noting at this point that there are stochastic reaction networks whose associated stochastic process undergoes explosions almost surely, but neverthless a probability distribution satisfying \eqref{cnstsol} exists. As an  example, consider the stochastic mass-action model implied by the following reaction graph (see Section \ref{sec:background} for the correspondence between the graph and the stochastic model),
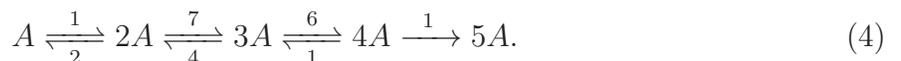
\begin{equation}\label{eq:example}
A\ce{<=>[1][2]}2A\ce{<=>[7][4]}3A\ce{<=>[6][1]}4A\ce{->[1]}5A.
\end{equation}
In Section \ref{sec:explosions}, we show that the transition rates implied by the above graph admit a distribution satisfying \eqref{cnstsol}, but a stochastic process satisfying these transition rates explodes almost surely in finite time.

Given a probability distribution satisfying \eqref{cnstsol}, rather than trying to address the explosion question by analyzing the forward equation we introduce a stochastic equation which we claim the corresponding process must satisfy. Let $\{N^{xy}:x,y\in E,x\neq y\}$ be independent Poisson processes where $N^{xy}$ has intensity $q(x,y)$. Then for each function in the collection 
\[\{f:E\rightarrow {\Bbb R}:\#\{x\in E:f(x)\neq 0\}<\infty \},\]
we require
\begin{equation}
f(X_t)=f(X_0)+\sum_{x,y\in E,x\neq y}\int_0^t(f(y)-f(X_{s-})){\bf 1}_{\{X_{s-}=x\}}dN^{xy}_s.\label{steq0}
\end{equation}
Note that we use functions $f$ with finite support in defining the equation rather than simply taking $f(x)=x$ to ensure that for each $f$, the right side of the equation has finitely many jumps in a finite time interval. To be convinced that this equation captures the intuition of transition intensities, note that if $X_t=x$, the probabilty that $X_{t+h}=y$ is 
\[
	P(N_{t+h}^{xy}-N_t^{xy}=1)+o(h)=q(x,y)he^{-q(x,y)h}+o(h)=q(x,y)h+o(h).
	\]

In the Appendix, we make the above claim precise by arguing that for each solution $(X_t)_{t\ge 0}$ of the stochastic equation \eqref{steq0}, the one dimensional distributions $P_t(x)=P(X_t=x)$ give a solution of \eqref{feq0}, and for each solution $(P_t)_{t \ge 0}$ of \eqref{feq0}, there exists a solution $(X_t)_{t\ge 0}$ of the stochastic equation such that $P_t(x)=P(X_t=x)$.   In particular,  in the case of a solution of \eqref{cnstsol}, there exists a solution $(X_t)_{t\ge 0}$ of \eqref{steq0} which is a stationary process with $P(X_t=x)=\pi (x)$ for all $t\geq 0$ and $x\in E$.

If $(X_t)_{t\ge 0}$ is a solution of \eqref{steq0}, then the number of jumps by $(X_t)_{t\ge 0}$ in the time interval $[0,t]$ is simply
\[J_X(t)=\sum_{x,y\in E,x\neq y}\int_0^t{\bf 1}_{\{X_{s-}=x\}}dN^{xy}_s.\]
It turns out (Theorem \ref{thm:main}) that we can resolve the original question by simply checking that for the stationary process
\[E[J_X(t)]=E[\int_0^t\sum_{x,y\in E,x\neq y}{\bf 1}_{\{X_{s-}=x\}}dN^{xy}_
s]=t\sum_{x,y\in E,x\neq y}\pi (x)q(x,y)<\infty .
\]
That is, the expected number of jumps per unit time of the stationary processes is finite.

It is worth mentioning that deciding whether a stochastic mass-action system undergoes an explosion is not a trivial task. We will demonstrate this via examples in Section \ref{sec:explosions}. In particular, knowledge of the behavior of the associated deterministic model (which is well known for complex balanced systems) is not enough to draw conclusions about the possibility of an explosion for the stochastic model: we will show explosive stochastic models whose deterministic counterpart has bounded solutions, and positive recurrent stochastic models whose deterministic counterpart has solutions which blow-up. 

\section{Notation}
\label{sec:background}

We start by introducing some notation we will use throughout: $\ZZ_{\ge 0}$ will denote the non-negative integers, $|\cdot|$ denotes the euclidean norm. Moreover, we will say that a vector $v$ is positive if all its entries are strictly positive. We will further use $X\sim \pi$ to mean  ``the probability distribution of the random variable $X$ is $\pi$'', and we define as support of a probability measure $\pi$ the set of states $x$ such that $\pi(x)>0$. Finally, we will denote by $\lim_{t\uparrow T}f(t)$ the limit of the function $f(t)$ for $t$ approaching $T$ from the left.

For notational convenience, we will use the following shorthand notations: for any two vectors $v,w\in\RR_{\ge 0}^n$ and any vector $x\in\ZZ_{\ge 0}^n$ we denote
\[
v^w=\prod_{i=1}^n (v_i)^{w_i}\quad\text{and}\quad x!=\prod_{i=1}^n (x_i)!,
\]
with the convention $0^0=1$.

A standard knowledge of continuous-time Markov chains is assumed. See for example \cite{norris:markov} for a more detailed introduction to the topic. In particular, terminology such as ``positive recurrence,'' ``transience,''  ``embedded discrete-time Markov chain,'' etc., will be used without being defined.

\section{Stationary distributions and explosions}\label{sec:explosions}

Let $(X_t)_{t\geq0}$ be a stochastic process on $E \subset \ZZ^d$ satisfying \eqref{steq0}. We  make the following technical assumptions on the intensities.

\begin{condition}\label{intcnd}
We require the following.
\begin{itemize}
\item For $x,y\in E$, $x\neq y$, $0\leq q(x,y)<\infty$.
\item For each $y\in E$, $\sum_{x\in E\setminus\{y\}}q(y,x)<\infty$.
\item For each $x\in E$, $\sup_{y\in E\setminus\{x\}}q(y,x)<\infty$.
\end{itemize}
\end{condition}

In the Appendix we will show that the one dimensional distributions of a process $(X_t)_{t\ge0}$ satisfying \eqref{steq0} and Condition \ref{intcnd} necessarily satisfy the forward Kolmogorov equation, that is \eqref{feq0} holds.

The following definition is standard \cite{norris:markov}.
\begin{definition}\label{def:explosion}
 Consider the continuous-time Markov chain $(X_t)_{t\geq0}$, and let $\tau_i$ be the time of the $i$th jump of the process. We say that $(X_t)_{t\geq0}$ undergoes an explosion if 
 $$T_\infty=\lim_{i\to\infty} \tau_i<\infty.$$
 If $P(T_\infty<\infty)>0$ given some initial distribution of $X_0$ on $E$, then we say that the process $(X_t)_{t\geq0}$ is \emph{explosive}.
\end{definition}
Note that if $(X_t)_{t\geq0}$ undergoes an explosion, then we necessarily have 
$$\limsup_{t\uparrow T_\infty}|X_t|=\infty.$$
To verify this last statement, suppose the last equation did not hold.  Then due to Condition \ref{intcnd} we would have 
\begin{equation}\label{eq:rates_bounded}
 \theta=\limsup_{t<T_\infty}q(X_t)<\infty,
\end{equation}
where
 $$q(x)=\sum_{y\in E\setminus\{x\}} q(x,y).$$
 Equation \eqref{eq:rates_bounded} would then imply that the expectations between jump times (which are exponentially distributed with state dependent rates given by $q(\cdot)$) are uniformly bounded from below by $1/\theta$.  Hence, the possibility of an explosion would be precluded.  See also \cite[Proposition 10.21]{kallenberg}.

We present a simple example that demonstrates the possibility of an explosion.

\begin{example}\label{pb}
Let $E=\{1,2,\ldots \}$, and define $q(x,x+1)=x^2$ and $q(x,y)=0$ otherwise. Starting with $x=1$ and proceeding by induction, it is simple to check the the solution of \eqref{feq0} is unique. It also follows that if $P_0(x)\geq 0$ for all $x\in E$ and $\sum_{x\in E}P_0(x)=1$, then $\sum_{x\in E}P_t(x)<1$ for all $t>0$. Consequently, if we want these intensities to determine a process $(X_t)_{t\geq0}$, then we must allow $X_t$ to be somewhere other than in $E$.   \hfill $\triangle$
\end{example}

With Example \ref{pb}\ in mind, we introduce an additional state $\partial$ and require solutions of \eqref{feq0} to satisfy both $P_t(x)\geq 0$ and $\sum_{x\in E}P_t(x)\leq 1$. We then define $P_t(\partial)=1-\sum_{x\in E}P_t(x)$. For solutions of the stochastic equation \eqref{steq0}, we extend the domain of $f$ to $E\cup \{\partial \}$ by defining $f(\partial )=0$ and define $X_t=\partial$, if $X_t\notin E$.

If $(X_t)_{t\geq0}$ undergoes an explosion then we have $X_{T_\infty}=\partial$, and the transition rates may only determine the process uniquely up to $T_\infty$. There may be different ways to extend the definition of the process $(X_t)_{t\geq0}$ for $t\geq T_\infty$, so that \eqref{steq0} holds.

The following definition is standard.
\begin{definition}\label{def:stat_process}
 Let $(X_t)_{t\geq0}$ be a stochastic process. We say that $(X_t)_{t\geq0}$ is a \emph{stationary process} with stationary distribution $\pi$ if the distribution of the translated process $X_{h+\cdot}$ does not depend on $h$, and $X_t$ is distributed according to $\pi$ for any time $t\in[0,\infty)$.
 
 Furthermore, we say that $\pi$ is a stationary distribution for the process $(X_t)_{t\geq0}$ if by assuming $X_0\sim\pi$ it follows that $(X_t)_{t\geq0}$ is a stationary process with stationary distribution $\pi$.
\end{definition}

An \emph{irreducible closed set} is a set $\Gamma$ such that for any three states $x_1, x_2\in\Gamma$ and $x_3\notin\Gamma$, we have that $X_0=x_1$ implies $P(X_t=x_2)>0$ and $P(X_t=x_3)=0$ for any $t>0$. It follows from standard Markov chain theory that the support of a stationary distribution is a union of irreducible closed sets \cite{norris:markov}. Furthermore, if the stochastic process satisfies \eqref{steq0} then \eqref{feq0} holds, and a stationary distribution $\pi$ is necessarily a constant solution to the forward equation, that is \eqref{cnstsol} holds.

It is well know that if a continuous-time Markov chain is non-explosive, at most one distribution satisfying \eqref{cnstsol} exists for any irreducible closed set \cite{norris:markov}. The same does not necessarily hold if the process is explosive. We will illustrate this in the following example. 
\begin{example}
 Consider a stochastic process on $\ZZ$ satisfying \eqref{steq0} with the following transition rates:
\begin{align*}
 q(n,n+1)=4^n\quad&\text{if }n\geq0;\\
 q(n,n-1)=\frac{4^n}{2}\quad&\text{if }n\geq1;\\
 q(n,n+1)=\frac{4^{-n}}{2}\quad&\text{if }n\leq-1;\\
 q(n,n-1)=4^{-n}\quad&\text{if }n\leq0,
\end{align*}
Then, it can be checked that
$$\pi_1(n)=\frac{1}{3}\cdot\frac{1}{2^{|n|}}\quad\text{and}\quad \pi_2(n)=\begin{dcases}
                                                                           \frac{1}{3}\cdot\frac{1}{4^{n}}&\text{if }n\geq0;\\
                                                                           \frac{1}{3}\cdot\frac{2^{-n+1}-1}{4^{-n}}&\text{if }n\leq-1
                                                                          \end{dcases}$$
are two different distributions satisfying \eqref{cnstsol}, with support on the unique irreducible closed set, that is $\ZZ$. Hence, the process is explosive.  \hfill $\triangle$
\end{example}

\section{Reaction networks}
We introduce terminology related to the models that motivated this work: reaction networks.  For more on this topic see \cite{AK2011,AK2015}.

A reaction network is a triple $(\Sp, \C, \R)$, where $\Sp$ is a finite set of \emph{species}, $\C$ is a finite set of linear combinations of species on $\ZZ_{\ge 0}$, referred to as \emph{complexes}, and finally $\R$ is a finite subset of $\C\times\C$, whose elements are called \emph{reactions}. It is usual to assume that $(y,y)\notin\R$ for any complex $y\in\C$. Usually, the $k$th reaction $(y_k,y'_k)$ is denoted by $y_k\to y'_k$. The corresponding \emph{reaction vector} is given by
\begin{equation}\label{eq:zeta}
\zeta_k=y'_k-y_k,
\end{equation}
which represents the net gain of molecules given by the occurrence of the reaction.

As an example, for the reaction network
\begin{align}\label{eq:ACR}
\begin{split}
	\emptyset \rightleftarrows A, \qquad 
	A + B &\rightleftarrows 3B
	\end{split}
\end{align}
 we have
\begin{gather*}
 \Sp=\{A,B\},\quad\C=\{\emptyset, A, A+B,3B\}\quad\text{and}\\
 \R=\{\emptyset \to A, A \to \emptyset, A + B \to 3B, 3B \to A+ B\}.
\end{gather*}
Arbitrarily enumerating species $A$ and $B$ as the first and second species, respectively, the reaction vectors are
\begin{align}\label{eq:ordering}
\left[ \begin{array}{r} 1\\0\end{array}\right],\quad  \left[ \begin{array}{r} -1\\0\end{array}\right], \quad  \left[ \begin{array}{r} -1\\2\end{array}\right], \quad \text{and}\quad  \left[ \begin{array}{r} 1\\-2\end{array}\right].
\end{align}

For notational convenience, we assume that the cardinality of $\Sp$ is $d$ and the cardinality of $\R$ is $m$. Any complex $y$ in $\C$ can be naturally identified with a vector in $\ZZ^d$, whose $i$th entry is the coefficient of $y$ relative to the $i$th species.  Note that this identification is being utilized in \eqref{eq:zeta}.

The main mathematical models for describing the time evolution of a system of interacting species are referred to as \emph{deterministic reaction system} and \emph{stochastic reaction system}.

\subsubsection*{Deterministic reaction system} The deterministic model is used to describe the time evolution of the concentrations of the different species. For any $t\geq0$, let $z(t)\in\RR_{\ge 0}^d$ be the vector whose entries represent the concentrations of the different species at time $t$. We associate with each reaction $y_k\to y'_k$ a rate function $\lambda_k:\RR^d_{\ge 0} \to \RR_{\ge 0}$, which describes how fast a given reaction occurs, given a certain state of the system. A choice of such rate functions constitutes a \emph{deterministic kinetics} for the model. A reaction network endowed with a deterministic kinetics is called a \emph{deterministic reaction system}. The evolution of $z(t)$ is described by the system of ordinary differential equations
\begin{align*}
\frac{d}{dt}z(t)=\sum_{k=1}^{m} \zeta_k\lambda_{k}(z(t)).
\end{align*}

A popular choice of kinetics is given by  \emph{deterministic mass-action kinetics}, which corresponds to the hypothesis that the system is well stirred. In this case, the rate functions are given by
\begin{align*}
\lambda_k(z)=\kappa_k z^{y_k}
\end{align*}
for some positive constants $\kappa_k$, called \emph{rate constants}. A reaction network endowed with deterministic mass-action kinetics is termed a \emph{deterministic mass-action system}.
Usually, a deterministic mass-action system is introduced with a graph as in \eqref{eq:example} and \eqref{eq:ACR}, where  numbers are placed next to the arrows to denote the rate constants associated with the different reactions.  
The deterministic mass-action system  associated with network \eqref{eq:ACR} is
\begin{align*}
\frac{d}{dt} z(t) = \kappa_1 \left[ \begin{array}{r} 1\\0\end{array}\right] + \kappa_2 z_1(t)  \left[ \begin{array}{r} -1\\0\end{array}\right]+ \kappa_3 z_1(t)z_2(t)  \left[ \begin{array}{r} -1\\2\end{array}\right] + \kappa_4 z_2(t)^3  \left[ \begin{array}{r} 1\\-2\end{array}\right].
\end{align*}
where the enumeration of the rate constants agrees with the ordering given in \eqref{eq:ordering}.

\subsubsection*{Stochastic reaction system} The stochastic model is used when the counts of the molecules of the different species is of interest. In this case, the time evolution of the model is described by means of a continuous-time Markov chain. For any $t\geq 0$, let $X_t\in\ZZ_{\ge 0}^d$ be the vector whose entries represent the counts of the molecules of the different species at time $t$. We associate with each reaction $y_k\to y'_k$ an intensity function $\lambda_k:\ZZ^d_{\ge 0} \to \RR_{\ge 0}$, which describes the propensity of the given reaction to occur, given a certain state of the system. Note that the notation for such a function is the same as in the deterministic setting, even though the domain of the functions is different. A choice of intensity functions constitutes a \emph{stochastic kinetics} for the model. A reaction network endowed with a stochastic kinetics is called a \emph{stochastic reaction system}. The time evolution of $X_t$ is then  described by a continuous-time Markov chain, where the transition rate from the state $x$ to the state $y$ is given by
\begin{equation}\label{eq:rates}
q(x,y) =  \sum_{k:\zeta_k=y-x}\lambda_k(x).
\end{equation}

To prevent  species counts from becoming negative we require that $\lambda_k(x)>0$ only if $x_i\geq y_{k,i}$ for each $i \in \{1,\dots,d\}$. In the event of explosion we let $X_t=\partial$ for any $t$ greater than the time at which the explosion occurs. Note that with this choice the Markov property of the process is preserved, and \eqref{steq0} holds. Moreover, Condition \eqref{intcnd} is always satisfied for transition rates given by \eqref{eq:rates}, since the number of reactions is finite.

As in the deterministic setting, a popular choice of stochastic kinetics, called \emph{stochastic mass-action kinetics}, derives from the assumption that the system is well stirred. In this case,
$$\lambda_k(x)=\kappa_k \frac{x!}{(x-y_k)!}\mathbbm{1}_{\{x\geq y_k\}}.$$
for some positive constants $\kappa_k$, which are called \emph{rate constants} as in the deterministic model. A reaction network endowed with stochastic mass-action kinetics is termed a \emph{stochastic mass-action system}.  For some choice of rate constants, the intensity functions for the stochastic mass-action system associated with reaction network \eqref{eq:ACR} are
\begin{align*}
\lambda_1(x) = \kappa_1, \quad \lambda_2(x) = \kappa_2 x_1, \quad \lambda_3(x) = \kappa_3 x_1x_2, \quad \text{ and } \quad \lambda_4(x) = \kappa_4 x_2(x_2-1)(x_2-2),
\end{align*}
where the reactions are again enumerated according to the ordering given in \eqref{eq:ordering}.

Exactly as for deterministic mass-action systems, a stochastic mass-action system is often described by means of a graph as in \eqref{eq:example} and \eqref{eq:ACR}, where  numbers are placed besides the arrows to denote the rate constants associated with the different reactions. It is therefore necessary to make explicit whether the system is modeled deterministically or stochastically, unless it is clear from the context.

\subsubsection*{Explosions in the context of reaction networks} We discuss here the stochastic mass-action system presented in the Introduction, which provides an example of a stochastic reaction system for which a distribution satisfying \eqref{cnstsol} exists, and for which an explosion occurs almost surely.

\begin{example}\label{ex:introduction}
 Consider the mass-action system \eqref{eq:example}, stochastically modeled. For convenience, we repeat it here:
$$A\ce{<=>[1][2]}2A\ce{<=>[7][4]}3A\ce{<=>[6][1]}4A\ce{->[1]}5A.$$
The stochastic model corresponds to a birth and death process whose state space is the positive integers and whose birth and death rates are, respectively,
$$q(x,x+1)=x+7x(x-1)+6x(x-1)(x-2)+x(x-1)(x-2)(x-3)=x^4\quad\text{for any }x\geq1$$
and
$$q(x,x-1)=2x(x-1)+4x(x-1)(x-2)+x(x-1)(x-2)(x-3)=x^2(x-1)^2\quad\text{for any }x\geq2.$$
It can be easily seen that
$$\pi(x)=\frac{6}{\pi^2x^2}$$
is a distribution satisfying \eqref{cnstsol}, and in particular is detailed balanced. Nevertheless, the model is explosive: we can show this by proving that the embedded discrete time Markov chain is transient, which follows from the following calculation
\begin{align*}
\sum_{n=2}^{\infty}\prod_{x=2}^{n} \frac{q(x,x-1)}{q(x,x+1)} = \sum_{n=2}^{\infty}\prod_{x=2}^{n}\frac{x^2(x-1)^2}{x^4}=\sum_{n=2}^{\infty}\prod_{x=2}^{n}\frac{(x-1)^2}{x^2}=\sum_{n=2}^{\infty}\frac{1}{n^2} <\infty,
\end{align*}
and the results of Section 3.3 in \cite{Lawler95}.
Since the embedded discrete time Markov chain is transient and the model has a distribution satisfying \eqref{cnstsol}, it must be explosive \cite{norris:markov}.  \hfill $\triangle$
\end{example}

In Example \ref{ex:introduction}, the solution of the corresponding deterministic mass-action system also has a blow-up.  Indeed,  the differential equation governing the deterministic dynamics is 
$$\frac{d}{dt}z(t)= 2z(t)^3+5z(t)^2+z(t),$$
whose solutions have a vertical asymptote for any $z(0)>0$. Interestingly, knowing that the associated ODE model is explosive is not enough to conclude that the stochastic model is explosive or even transient. To support the last assertion, we give two additional examples. In Example \ref{ex:blowuprecurrent} all solutions of the deterministic model with positive initial condition have a blow-up, while the corresponding stochastic process is non-explosive. Conversely, in Example \ref{ex:regularexplosive} the ODE system has bounded solutions, while the associated stochastic model is explosive.

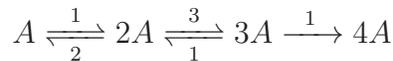
\begin{example}\label{ex:blowuprecurrent}
 Consider
$$ A\ce{<=>[1][2]}2A\ce{<=>[3][1]}3A\ce{->[1]}4A$$
As in Example \ref{ex:introduction}, the stochastic model corresponds to a birth and death process whose state space is the positive integers. The birth and death rates are, respectively,
$$q(x,x+1)=x+3x(x-1)+x(x-1)(x-2)=x^3\quad\text{for any }x\geq1$$
and
$$q(x,x-1)=2x(x-1)+x(x-1)(x-2)=x^2(x-1)\quad\text{for any }x\geq2.$$
The differential equation of the corresponding deterministic mass-action system is 
$$\frac{d}{dt}z(t)= z(t)^2+z(t),$$
so the deterministic model has a blow-up for any $z(0)>0$. Nevertheless, in this case the stochastic model is non-explosive and is in fact positive recurrent. We can check the recurrence of the continuous-time Markov chain by
\begin{align*}
\sum_{n=2}^{\infty}\prod_{x=2}^{n} \frac{q(x,x-1)}{q(x,x+1)} = \sum_{n=2}^{\infty}\prod_{x=2}^{n}\frac{x^2(x-1)}{x^3}=\sum_{n=2}^{\infty}\prod_{x=2}^{n}\frac{x-1}{x}=\sum_{n=2}^{\infty}\frac{1}{n}=\infty.
\end{align*}
and the results of Section 3.3 in \cite{Lawler95}. Then, positive recurrence follows from the fact that
$$\pi(x)=\frac{6}{\pi^2x^2}$$
satisfies \eqref{cnstsol},  as  can be easily checked.  Note that $\pi(\cdot)$ on the left of the above denotes the  distribution, whereas the $\pi$ on the right denotes the usual well-known constant. \hfill $\triangle$
\end{example}
\begin{example}\label{ex:regularexplosive}
 Consider the mass-action system
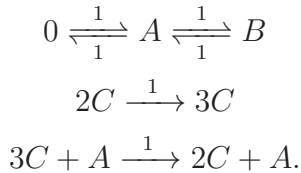
\begin{gather*}
 0\ce{<=>[1][1]}A\ce{<=>[1][1]}B\\
 2C\ce{->[1]}3C\\
 3C+A\ce{->[1]}2C+A.
\end{gather*}
Let $z(t)$ be the solution of the deterministic model, with $z_1(t)$, $z_2(t)$ and $z_3(t)$ denoting the concentration of the species $A$, $B$ and $C$, respectively. Then, $z(\cdot)$ converges to the unique positive equilibrium
$$(1,1,1)$$
for any initial condition $z(0)\in\RR_{\ge 0}^3$, with $z_3(0)>0$ (otherwise the solution would tend to $(1,1,0)$). Conversely, for any initial condition $X_0\in\ZZ_{\ge 0}^3$ with the counts of species $C$ greater than or equal to 2, the stochastic model undergoes an explosion almost surely: it occurs infinitely often that the counts of $A$ hit 0, and remain null for a positive time before either $0\to A$ or $B\to A$ takes place. Whenever the counts of $A$ are 0, the counts of $C$ can only change through the reaction $2C\to 3C$. That is, the counts of $C$ change accordingly to a pure birth process with birth rate
$$q(x,x+1)=x(x-1),$$
which is known to explode almost surely (see \cite{norris:markov}). Hence, whenever the counts of $A$ are 0, the counts of $C$ have a positive probability of exploding. Furthermore, it can be shown that the network $0\ce{<=>}A\ce{<=>}B$ defines a positive recurrent Markov chain: indeed, the associated model corresponds to two $M/M/\infty$ queues linked to each other. Hence, the counts of the molecules $A$ hit 0 infinitely often, and an explosion of the counts of $C$ will occur almost surely.\hfill $\triangle$
\end{example}

In conclusion, determining when a stochastic mass-action system undergoes an explosion remains in general a hard question. In particular, studying the behavior of the associated deterministic model is in general not enough.

\subsubsection*{Complex balancing}

In the context of deterministic reaction systems, there exists a particular kind of equilibria, referred to as \emph{complex balanced}, whose properties have been intensively studied. The results in the context of chemical reaction network theory date back to at least \cite{horn:general, feinberg}. Indeed, complex balanced equilibria are a generalization of detailed balanced equilibria, which are fundamental in the study of thermodynamics. We give the formal definition here.
\begin{definition}\label{def:cb}
 We say that $c\in \RR^d_{\ge 0}$ is a \emph{complex balanced} equilibrium if for any $y\in\C$ we have
 \begin{equation}\label{eq:cb_general}
  \sum_{k:y_k=y}\lambda_k(c)=\sum_{k:y'_k=y}\lambda_k(c),
 \end{equation}
 where the sum on the left, respectively right, is over those reactions with $y$ as source, respectively product, complex.
\end{definition}
Namely, at complex balanced equilbria the flux ``entering'' each complex equals the flux ``exiting'' from it. For mass-action kinetics, \eqref{eq:cb_general} becomes
$$\sum_{k:y_k=y}\kappa_kc^{y_k}=\sum_{k:y'_k=y}\kappa_kc^{y_k}.$$
Complex balanced equilibria have particularly nice properties in the context of mass-action systems. For example, positive complex balanced equilibria are locally asymptotically stable, and if a positive complex balanced equilibrium exists, then all equilibria are necessarily complex balanced \cite{horn:general}. Due to the last property, we can give the following definition:
\begin{definition}
 We say that a deterministic mass-action system is \emph{complex balanced} if one (and therefore all) positive equilibrium of the system is complex balanced, and at least one positive equilibrium exists.
\end{definition}
For example, note that the deterministic mass-action system associated with the network \eqref{eq:ACR} with rate constants $\kappa_1=\kappa_2=\kappa_3 = \kappa_4=1$ is complex balanced with  equilibrium  $c = (1,1)$.  Note also that there can not be a choice of rate constants for the model \eqref{eq:example} which admit a complex balanced equilibrium as it is impossible to ever balance the complex $5A$.

The relevance of complex balancing in the setting of stochastic reaction systems was first explored in \cite{ACK2010}, where the following result is shown.  Followup results include \cite{AC2016,cappelletti:complex_balanced}.
\begin{theorem}\label{thm:poisson}
 If a deterministic mass-action system with rate constants $\kappa_k$ is complex balanced and has a positive complex balanced equilibrium $c$, then the corresponding stochastically modeled system, with the same choice of rate constants, admits a solution to \eqref{cnstsol} of the form
 \begin{align}\label{eq:distribution}
 	\pi(x) =\prod_{i=1}^d \frac{c_i^x}{x_i!}e^{-c_i}, \quad x \in \ZZ^d_{\ge 0}.
 \end{align}
\end{theorem}
The main utility of the above theorem resides in the fact that a  reaction network will admit a complex balanced equilibrium if the network satisfies two easily checked properties: weak reversibility and a deficiency of zero \cite{feinberg}.  It is straightforward to check that the network \eqref{eq:ACR} is weakly reversible and has a deficiency of zero.  Therefore, we know that it admits a distribution satisfying \eqref{cnstsol} of the form \eqref{eq:distribution}.  (Though to find the distribution on a particular irreducible closed set of the state space, \eqref{eq:distribution} needs to be projected onto it. See \cite{SZ2010} for more on this.)

In \cite{cappelletti:complex_balanced} the converse of Theorem \ref{thm:poisson} is proven. Also in \cite{cappelletti:complex_balanced} a definition of complex balancing in the setting of stochastic reaction systems is provided and it is proven that a stochastic mass-action system is complex balanced if and only if it is complex balanced when deterministically modeled. 

\section{Main results}

In this section, we consider a stochastic process $(X_t)_{t\geq0}$ satisfying \eqref{steq0}, with transition rates satisfying Condition \ref{intcnd}. We then consider a distribution $\pi$ which is a constant solution to the forward equation, that is $\pi$ satisfies \eqref{cnstsol}. 
Here we prove that if 
\begin{equation}\label{eq:expectation_finite}
  \sum_{x\in\ZZ^d_{\ge 0}}\pi(x)q(x)<\infty
\end{equation}
 (recall $q(x)=\sum_{y\in E\setminus\{x\}} q(x,y)$), then $(X_t)_{t\geq0}$ is non-explosive. In turn, positive recurrence of $(X_t)_{t\geq0}$ within the irreducible closed sets where $\pi$ is positive is assured. Finally, by standard Markov chain theory, it follows that $\pi$ describes  the limit distribution of the model restricted to irreducible closed sets where $\pi$ is positive.

We state here our main result.
  
\begin{theorem}\label{thm:main}
Consider a stochastic process $(X_t)_{t\geq0}$ on a countable state space $E\subset \ZZ^d$, satisfying \eqref{steq0} and whose rates satisfy Condition \ref{intcnd}. Assume there exists a distribution $\pi$ satisfying \eqref{cnstsol}, such that \eqref{eq:expectation_finite} holds. Then, the process $(X_t)_{t\geq0}$ is non-explosive if the support of $P_0$ is contained in the support of $\pi$.
\end{theorem}
\begin{proof}
 In the Appendix, we will show that under the conditions of Theorem \ref{thm:main} a stationary process $(\tilde{X}_t)_{t\geq0}$ satisfying \eqref{steq0} and with stationary distribution $\pi$ exists. Note that, as already discussed in the Introduction, \eqref{eq:expectation_finite} states that the expected number of jumps of $(\tilde{X}_t)_{t\geq0}$ per unit time is finite. Hence, $(\tilde{X}_t)_{t\geq0}$ necessarily has finitely many jumps in compact time intervals, and is non-explosive. It follows that any process with the support of the initial distribution contained in the support of $\pi$ is non-explosive, and this concludes the proof.
\end{proof}

\begin{remark}\label{rem:embedded}
 It is worth noting that \eqref{eq:expectation_finite} holds if and only if the discrete time embedded process of $(X_t)_{t\geq0}$ is positive recurrent. We will show it here.
 
 Up to the potential explosion at time $T_\infty$, $(X_t)_{t\geq0}$ is a continuous-time Markov chain. Let $\{Z_n\}_{n\in\ZZ_{\ge 0}}$ be the embedded discrete time Markov chain of $(X_t)_{0\leq t<T_{\infty}}$. The transition probabilities for $Z_n$ are given by
\[
	p(x,y) = \frac{q(x,y)}{q(x)}.
\]
Set 
\begin{align}
\label{eq:66}
\pi_Z(x) = \pi(x)q(x).
\end{align}
Then, $\pi_Z$ can be normalized to be a stationary distribution for $Z_n$. To show this, we need to show that for any $x \in E$,
\begin{align}
\label{eq:5678988}
 \pi_Z(x) = \sum_{y\in E} \pi_Z(y) p(y,x),
\end{align}
and that
\begin{align}
\label{eq:545}
\sum_{x\in\ZZ^d_{\ge 0}} \pi_Z(x) < \infty.
\end{align}
 
 Plugging $\pi_Z$ from \eqref{eq:66} into the right-hand side of \eqref{eq:5678988} yields
$$
	 \sum_{y\in E} \pi_Z(y) p(y,x) = \sum_{y\in E} \pi(y) q(y,x) = \pi(x)q(x),
$$
where the final equality  holds since $\pi$ satisfies \eqref{cnstsol}.  Due to \eqref{eq:66} we have what we want.

Next, 
$$
	\sum_{x\in E} \pi_Z(x) = \sum_{x\in E} \pi(x) q(x) <\infty
$$
by assumption. Hence, all the points in the support of $\pi$ must be recurrent for the process $(Z_n)_{n\in\ZZ_{\ge 0}}$.
 \end{remark}
 
We can specialize to the case of reaction network theory, in which case Theorem \ref{thm:main} becomes the following.

\begin{corollary}\label{cor:main}
 Let $(X_t)_{t\geq0}$ be the continuous-time Markov chain associated with a stochastic reaction system. Assume that a distribution $\pi$ satisfying \eqref{cnstsol} exists, such that 
 \begin{equation}\label{eq:condition}
  \sum_{x\in\ZZ^d_{\ge 0}} \pi(x)\sum_{k=1}^m\lambda_k(x)<\infty.
 \end{equation}
Then, the process $(X_t)_{t\geq0}$ is non-explosive if the support of $P_0$ is in the support of $\pi$.
\end{corollary}

The proof relies on Theorem \ref{thm:main} and on the fact that in the case of stochastic reaction systems we have
$$q(x)=\sum_{k=1}^m \lambda_k(x).$$
Note that Corollary \ref{cor:main} provides sufficient conditions for non-explosivity. However these conditions are not necessary. For example, consider the model associated to the mass-action system in Example \ref{ex:blowuprecurrent}, which we have already proved  to be non-explosive. In this case
$$\sum_{k=1}^m\sum_{x\in\ZZ^d_{\ge 0}}\pi(x)\lambda_k(x)=\sum_{x=0}^\infty \frac{6(2x^3 - x^2)}{\pi^2 x^2} =\infty.$$
Since the model is positive recurrent and the state space is irreducible, it can be further shown that
$$\pi(x)=\frac{6}{\pi^2x^2}$$
is the unique distribution satisfying \eqref{cnstsol}. Hence, in this case, \eqref{eq:condition} does not hold for any distribution satisfying \eqref{cnstsol}. In the spirit of Remark \ref{rem:embedded}, it is worth noting that in this case, the embedded discrete time Markov chain is null recurrent.

The following is a consequence of Corollary \ref{cor:main}, which for practical purposes is the main result of the paper.

\begin{theorem}
For any initial distribution, complex balanced stochastic mass-action systems are non-explosive.
\end{theorem}
\begin{proof}
 By Theorem \ref{thm:poisson}, the state space of a complex balanced stochastic mass-action system is a union of irreducible closed sets, and a distribution satisfying \eqref{cnstsol} of the form
 $$\pi(x)=M\frac{c^x}{x!}\quad\text{for any }x\in\ZZ_{\ge 0}^d$$
 exists, for a suitable positive vector $c$ and a finite normalizing constant $M$. Hence, the proof follows from  Corollary \ref{cor:main}, as 
 $$\sum_{x\in\ZZ^d_{\ge 0}}\pi(x)\lambda_k(x)=M\sum_{x\in\ZZ^d_{\ge 0}}\frac{c^x}{x!}\lambda_k(x)<\infty,$$
 since under mass-action kinetics the rates $\lambda_k$ are polynomials with fixed degree.
\end{proof}

\section*{Appendix}

Let $E\subset \ZZ^d$ and for $x,y\in E$, $x\neq y$, let $q(x,y)\geq 0$ and assume Condition \ref{intcnd} is satisfied. As already discussed in the Introduction, a continuous-time Markov chain on $E$ with transition intensities $\{q(x,y)\}$ is intuitively a stochastic process $(X_t)_{t\geq0}$ in $E$ such that 
\[P(X_{t+h}=y|{\cal F}_t^X)=q(X_t,y)h+o(h),\quad t,h\geq 0,\]
where $\{{\cal F}_t^X\}$ is the filtration generated by $(X_t)_{t\geq0}$.  There are several ways of making this intuition precise. 

\begin{definition}\label{def:appendix}
\begin{enumerate}[a)]
\item The one-dimensional distributions of $(X_t)_{t\geq0}$ satisfy the {\em forward\/} or {\em master equation\/} if for every $x\in E$
\begin{equation}\label{feq}
 \frac{d}{dt}P_t(x)=\sum_{y\in E\setminus\{x\}} P_t(y)q(y,x)-\sum_{y\in E\setminus\{x\}} P_t(x)q(x,y),
\end{equation}
where the derivative at $t=0$ is interpreted as a right derivative.
\item Consider 
$${\cal D}(E)=\{f\colon E\to\RR: \#\{x\in E:f(x)\neq 0\}<\infty \}$$
and define the linear operator $A$ on ${\cal D}(E)$ by
\[Af(x)=\sum_{y\in E\setminus\{x\}}q(x,y)(f(y)-f(x)).\]
Then $(X_t)_{t\geq0}$ is a solution of the {\em martingale problem\/} for $A$ if
\begin{equation}\label{mgf}
 f(X_t)-f(X_0)-\int_0^tAf(X_s)ds
\end{equation}
is a $\{{\cal F}_t^X\}$-martingale for each $f\in {\cal D}(E)$.
\item For $x,y\in E$, $x\neq y$, let $(N^{xy}_t)_{t\geq0}$ be a Poisson process with intensity $q(x,y)$. The {\em stochastic equation\/} for $(X_t)_{t\geq0}$ is given by
\begin{equation}\label{eq:stochequation88}
f(X_t)=f(X_0)+\sum_{x,y\in E :x\neq y}\int_0^t(f(y)-f(X_{s-})){\bf 1}_{\{X_{s-}=x\}}dN^{xy}_s.
\end{equation}
\end{enumerate}
\end{definition} 

\begin{remark}
Note that by Condition \ref{intcnd}, for $f\in {\cal D}(E)$, $\sup_{x\in E}|Af(x)|<\infty$.
\end{remark}

Furthermore, we introduce an additional state $\partial$ as done in the main text. Solutions of \eqref{feq} need to satisfy $P_t(x)\geq 0$, $\sum_{x\in E}P_t(x)\leq 1$ and $P_t(\partial)=1-\sum_{x\in E}P_t(x)$. Moreover, for solutions of the martingale problem and the stochastic equation, we extend $f\in {\cal D}(E)$ to $E\cup \{\partial \}$ by defining $f(\partial )=0$ and define $X_t=\partial$, if $X_t\notin E$.

With this understanding of solutions, it follows that any solution of the stochastic equation \eqref{eq:stochequation88} is a solution of the martingale problem. Indeed, we can write
$$f(X_t)-f(X_0)-\int_0^tAf(X_s)ds=\sum_{x,y\in E :x\neq y}\int_0^t(f(y)-f(X_{s-})){\bf 1}_{\{X_{s-}=x\}}d\tilde{N}^{xy}_s,$$
where
$$\tilde{N}^{xy}_s = N^{xy}_s-q(x,y)s$$
is a martingale. Moreover, for any solution of the martingale problem, $P_t(x)=P(X_t=x)$ is a solution of \eqref{feq}: this follows by considering $f={\bf 1}_{\{x\}}$ for $x\in E$ and by taking expectation \cite{AK2015}. The converse of these observations also holds.

\begin{theorem}\label{thm:appendix}
If $(P_t)_{t\geq0}$ is a solution of \eqref{feq}, then there exists a solution of the martingale problem such that for $x\in E\cup\partial$, $P(X_t=x)=P_t(x)$.  If $(X_t)_{t\geq0}$ is a solution of the martingale 
problem, then there exists a solution $(\tilde{X}_t)_{t\geq0}$ of the stochastic equation such that $(X_t)_{t\geq0}$ and $(\tilde{X}_t)_{t\geq0}$ have the same distribution.
\end{theorem}

\begin{proof}
The first statement follows from Corollary 3.2 of \cite{KS98}. The second statement can be proved by arguments similar to the proof of (17) in \cite{Kur11}.
\end{proof}

It follows from Theorem \ref{thm:appendix} that weak uniqueness for any of the three characterizations implies weak uniqueness for the other two, for a fixed initial distribution. Moreover, the law of a solution to the stochastic equation is uniquely determined up to time $T_{\infty}$, where $T_{\infty}$ is defined as in Definition \ref{def:explosion}. This observation provides the proof of the next result.

\begin{theorem}
 If for a fixed initial distribution $P_0$ we have $P(T_{\infty}=\infty)=1$, then there exists a unique process $(X_t)_{t\geq0}$ with initial distribution $P_0$ satisfying any of the three characterization of Definition \ref{def:appendix}.
\end{theorem}

We conclude with the following result, which can be read from \cite[Theorem 9.17 of Chapter 4]{ethier_kurtz} and is due to \cite{echeverria}.
\begin{theorem}\label{thm5}
 If $\pi$ is a constant solution to the forward Kolmogorov equation, namely if
$$ \sum_{y\in E\setminus\{x\}} \pi(x)q(x,y)=\sum_{y\in E\setminus\{x\}} \pi(y)q(y,x),$$
then there exists a solution of the martingale problem which is a stationary process with stationary distribution $\pi$.
\end{theorem}

Note that under the hypotheses of Theorem \ref{thm5}, Theorem \ref{thm:appendix} implies the existence of a solution to the martingale problem with $P_t=\pi$ for all $t\geq 0$, and a stationary solution to the stochastic equation \eqref{eq:stochequation88}.

 \bibliographystyle{plain}
\bibliography{nonexplosivity}

\end{document}